\documentclass[12pt, reqno]{amsart}

\author[P.~Leonetti]{Paolo Leonetti}
\address{Department of Statistics, Universit\`a Bocconi, via Roentgen 1, Milan 20136, Italy}
\email{leonetti.paolo@gmail.com}
\urladdr{\url{https://sites.google.com/site/leonettipaolo/}} 

\title{Tauberian theorems for ordinary convergence}

\usepackage[T1]{fontenc}
\usepackage{amsmath}
\usepackage{amssymb}
\usepackage{amsthm}
\usepackage[left=2.5cm, right=2.5cm, top=3cm]{geometry}
\usepackage{hyperref}
\usepackage{fancyhdr}
\usepackage[inline]{enumitem}
\usepackage{comment}
\usepackage{nicefrac}
\usepackage{bm}
\usepackage{mathrsfs}
\usepackage{graphicx}
\usepackage[utf8]{inputenc}
\usepackage{cancel}
\usepackage{mathtools}

\newcommand{\vertiii}[1]{{\left\vert\kern-0.25ex\left\vert\kern-0.25ex\left\vert #1 
    \right\vert\kern-0.25ex\right\vert\kern-0.25ex\right\vert}}
		
\AtBeginDocument{%
   \def\MR#1{}
}

\newtheorem{thm}{Theorem}[section]
\newtheorem{cor}[thm]{Corollary}
\newtheorem{lem}[thm]{Lemma}

\theoremstyle{definition} 
\newtheorem{defi}[thm]{Definition}
\let\olddefi\defi
\renewcommand{\defi}{\olddefi\normalfont}
\newtheorem{example}[thm]{Example}
\let\oldexample\example
\renewcommand{\example}{\oldexample\normalfont}
\newtheorem{rmk}[thm]{Remark}
\let\oldrmk\rmk
\renewcommand{\rmk}{\oldrmk\normalfont}

\hypersetup{
    pdftitle={Tauberian theorems for ordinary convergence},
    pdfauthor={Paolo Leonetti},
    pdfmenubar=false,
    pdffitwindow=true,
    pdfstartview=FitH,
    colorlinks=true,
    linkcolor=blue,
    citecolor=green,
    urlcolor=cyan
}

\providecommand{\MR}[1]{}

\providecommand{\MR}{\relax\ifhmode\unskip\space\fi MR }

\providecommand{\href}[2]{#2}

\begin{document}

\begin{abstract}
\noindent 
We show that 
a real sequence $x$ is convergent if and only if there exist a regular matrix $A$ and an $F_{\sigma\delta}$-ideal $\mathcal{I}$ on $\mathbf{N}$ such that the set of subsequences $y$ of $x$ for which $Ay$ is $\mathcal{I}$-convergent is of the second Baire category. 
This includes the cases where $\mathcal{I}$ is the ideal of asymptotic density zero sets, the ideal of Banach density zero sets, and the ideal of finite sets.
The latter recovers an old result given by Keogh and Petersen in [J. London Math. Soc.  \textbf{33} (1958), 121--123]. 
Our proofs are of a different nature and rely on recent results 
in the context of 
$\mathcal{I}$-Baire classes and filter games. 

As application, we obtain a stronger version of the classical Steinhaus' theorem: for each regular matrix $A$, there exists a $\{0,1\}$-valued sequence $x$ such that $Ax$ is not statistically convergent. 
\end{abstract}

\maketitle
\thispagestyle{empty}

\section{Introduction}


Given an infinite real matrix $A=(a_{n,k})$, we say that $A$ sums a real sequence $x$ if $Ax=(\sum_k a_{n,k}x_k)$ is well defined and convergent. 
The matrix $A$ is regular if it maps convergent sequences into convergent sequences, preserving the corresponding limits. 

The classical Steinhaus' theorem states that a regular matrix $A$ cannot sum all $\{0,1\}$-valued sequences, see e.g. \cite{MR795251}. This can be rewritten equivalently as follows: a regular matrix $A$ cannot sum all subsequences of a divergent $\{0,1\}$-valued sequence $x$. We recall also that Hahn's theorem \cite[Theorem 2.4.5]{MR1817226} states that a matrix $A$ sums all $\{0,1\}$-valued sequences if and only if it sums all bounded sequences. On this direction, Buck \cite{MR9208, MR81983} showed that a regular matrix $A$ cannot sum all subsequences of a given divergent sequence $x$. 

Let $\Sigma$ be the set of strictly increasing functions $\sigma: \mathbf{N}\to \mathbf{N}$, so that a subsequence of a sequence $x$ is uniquely identified by some $\sigma \in \Sigma$, and write $\sigma(x):=(x_{\sigma(n)})$. Accordingly, Buck's result tells us that, if $A$ is regular and $x$ is divergent, then the set
\begin{equation}\label{eq:definitionsigmaxA}
\Sigma_{x,A}:=\{\sigma \in \Sigma: A\sigma(x) \in c\}
\end{equation}
cannot be equal to $\Sigma$. Finally, Keogh and Petersen \cite{MR92024} proved the following:
\begin{thm}\label{thm:keoghpetersen}
Let $x$ be a divergent sequence and $A$ be a regular matrix. Then $\Sigma_{x,A}$ is meager. 
\end{thm}
This is meaningful since $\Sigma$ is a $G_\delta$-subset of the Polish space $\mathbf{N}^{\mathbf{N}}$, hence it is a Polish space as well by Alexandrov's theorem; in particular, $\Sigma$ is not meager in itself. 
Special instances of Theorem \ref{thm:keoghpetersen} have been recently rediscovered in \cite{MR4049166, MR3809152, MR3216022}. Related results can be found in \cite[Theorem 2.3]{MR13435} and \cite{MR404918, MR425410, MR261220, MR1123336}.

Note that, if $x$ is convergent and $A$ is regular, then $\Sigma_{x,A}$ is equal to the whole $\Sigma$. Therefore, by Theorem \ref{thm:keoghpetersen}, $x$ is convergent if and only if there exists a regular matrix $A$ such that $\Sigma_{x,A}$ is not meager. 
The aim of this work is to show that Theorem \ref{thm:keoghpetersen} holds even if the space of convergent sequences $c$ in the definition of $\Sigma_{x,A}$ in \eqref{eq:definitionsigmaxA} is replaced by a much bigger set, e.g., the space of statistically convergent sequences (see Section \ref{sec:main}). 
Our methods are different from those employed by Keogh and Petersen, and rely on more recent result on filter games and ideal Baire classes. 


\section{Main result}\label{sec:main}

Let $\mathcal{I}$ be an ideal on the positive integers $\mathbf{N}$, that is, a hereditary family of subsets of $\mathbf{N}$ closed under finite unions; 
moreover, it is assumed that $\mathcal{I}$ contains the family of finite sets $\mathrm{Fin}$ and it is different from $\mathcal{P}(\mathbf{N})$. 
Let 
$\mathcal{I}^\star:=\{S\subseteq \mathbf{N}: S^c\in \mathcal{I}\}$ be its dual filter. 
Every subset of $\mathcal{P}(\mathbf{N})$ is endowed with the relative Cantor-space topology. 
In particular, we may speak about $F_\sigma$-subsets of $\mathcal{P}(\mathbf{N})$, meager ideals, etc. 
In addition, let $\mathcal{Z}$ be the ideal of asymptotic density zero sets, that is, 
$$
\textstyle 
\mathcal{Z}:=\left\{S\subseteq \mathbf{N}: \lim_n |S\cap [1,n]|/n=0\right\}.
$$
For our purposes, it will be important also to recall the definition of the ideal $\mathrm{Fin}\times \mathrm{Fin}$, which can be represented as an ideal on $\mathbf{N}$ by
$$
\textstyle 
\mathrm{Fin}\times \mathrm{Fin}:=\left\{S\subseteq \mathbf{N}: \forall^\infty k\in \mathbf{N}, \{n \in S: \nu_2(n)=k\}\in \mathrm{Fin}\right\},
$$
where $\nu_2(n)$ stands for the $2$-adic valuation of $n\in \mathbf{N}$ (that is, $\nu_2(n)=\max\{k\ge 0: 2^k \text{ divides }n\}$). 
We recall that $\mathcal{Z}$ is a $F_{\sigma\delta}$-ideal, while $\mathrm{Fin}\times \mathrm{Fin}$ is $F_{\sigma\delta\sigma}$. 
Given ideals $\mathcal{I}, \mathcal{J}$ on $\mathbf{N}$, we say that $\mathcal{I}$ contains an isomorphic copy of $\mathcal{J}$ if there exists an injection $\iota: \mathbf{N} \to \mathbf{N}$ such that $S \in \mathcal{J}$ if and only if $\iota^{-1}[S] \in \mathcal{I}$ for all $S\subseteq \mathbf{N}$; see e.g. \cite{MR1711328}. 

Let $\omega$ be the vector space of all real sequences, together with its classical subspaces $\ell_\infty$, $c$, $c_0$, $\ell_1$, $c_{00}$ of bounded, convergent, convergent to $0$, absolutely summable, and eventually $0$ sequences, respectively. 
In addition, given an ideal $\mathcal{I}$ on $\mathbf{N}$, we define also the following subspaces of $\omega$: 
\begin{displaymath}
\begin{split}
\ell_\infty(\mathcal{I}) &:=\{\,\mathcal{I}\text{-bounded sequences }\}, \\
c(\mathcal{I})& :=\{\,\mathcal{I}\text{-convergent sequences }\}, \\
c_b(\mathcal{I})&:= \{\,\mathcal{I}\text{-convergent bounded sequences }\}. \\
\end{split}
\end{displaymath}
Note that $c=c(\mathrm{Fin})=c_b(\mathrm{Fin})$ 
and $\ell_\infty=\ell_\infty(\mathrm{Fin})$. Here, a sequence $x \in \omega$ is $\mathcal{I}$-bounded if there exists $k$ such that $\{n\in \mathbf{N}: |x_n|\ge k\} \in \mathcal{I}$. 
Moreover, $x$ is $\mathcal{I}$-convergent to $\eta \in \mathbf{R}$, shortened as $\mathcal{I}\text{-}\lim x=\eta$, provided that $\{n \in \mathbf{N}: |x_n-\eta|>\varepsilon\} \in \mathcal{I}$ for all $\varepsilon>0$. 
Every subspace of $\ell_\infty$ will be endowed with the topology induced by the supremum norm. 
In the literature, $\mathcal{Z}$-convergence is usually called \emph{statistical convergence}. 

At this point, for each infinite real matrix $A=(a_{n,k})$, let $\omega_A$ be the set of sequences $x$ such that $Ax$ is well defined, so that
$$
\textstyle 
\omega_A:=\left\{x \in \omega: \sum_ka_{n,k}x_k \text{ converges for all }n\right\}.
$$
For all nonempty $X,Y\subseteq \omega$, let $(X,Y)$ be the matrix class of matrices $A$ which map sequences in $X$ into sequences in $Y$, i.e.,
$
\textstyle 
(X,Y):=\{A: x \in \omega_A \text{ and } Ax \in Y \text{ for all }x \in X\}.
$ 
Given ideals $\mathcal{I}, \mathcal{J}$ on $\mathbf{N}$, a matrix $A$ is $\bm{(\mathcal{I}, \mathcal{J})}$\textbf{-regular} if it is maps $\mathcal{I}$-convergent bounded sequences into $\mathcal{J}$-convergent bounded sequences, preserving the corresponding ideal limits, that is,  
$$
\textstyle 
A \in (c_b(\mathcal{I}),c_b(\mathcal{J}))
\quad \text{ and } \quad 
\mathcal{I}\text{-}\lim x=\mathcal{J}\text{-}\lim Ax
\,\,\text{ for all }
 x \in c_b(\mathcal{I}).
$$
Note that $(\mathrm{Fin}, \mathrm{Fin})$-regular matrices are simply the classical regular matrices. 
Probably the most important regular matrix is the Cesàro matrix $C_1=(a_{n,k})$ defined by $a_{n,k}=\frac{1}{n}$ if $k\le n$ and $a_{n,k}=0$ otherwise. 
Classes of $(\mathcal{I}, \mathcal{J})$-regular matrices have been recently used and characterized in \cite{ConnorLeo}.
In particular, the following result extends the classical Silverman--Toeplitz characterization, see \cite[Theorem 1.2 and Theorem 1.3]{ConnorLeo}:
\begin{thm}\label{thm:connotleofinJ}
A matrix $A$ is $(\mathrm{Fin},\mathcal{I})$-regular if and only if:
\begin{enumerate}[label={\rm (\textsc{R}\arabic{*})}]
\item \label{item:R1} $\sup_n \sum_k|a_{n,k}|<\infty$\textup{;}
\item \label{item:R2} $\mathcal{I}\text{-}\lim_n a_{n,k}=0$ for all $k \in \mathbf{N}$\textup{;}
\item \label{item:R3} $\mathcal{I}\text{-}\lim_n \sum_k a_{n,k}=1$\textup{.}
\end{enumerate}
\end{thm}

Finally, for each real sequence $x$, matrix $A$, and ideal $\mathcal{I}$, define 
$$
\textstyle 
\Sigma_{x,A}(\mathcal{I}):=\left\{\sigma \in \Sigma: A\sigma(x) \in c(\mathcal{I})\right\}.
$$
Note that $\Sigma_{x,A}=\Sigma_{x,A}(\mathrm{Fin})$. 
Our main result follows (for the proof, see Section \ref{sec:mainproof}):
\begin{thm}\label{thm:mainTauberianFinI}
Let $x$ be a divergent sequence, $\mathcal{I}$ be a Borel ideal on $\mathbf{N}$ which does not contain an isomorphic copy of $\mathrm{Fin}\times \mathrm{Fin}$, and $A$ be a $(\mathrm{Fin}, \mathcal{I})$-regular matrix. Then $\Sigma_{x,A}(\mathcal{I})$ is meager. 
\end{thm}

Some remarks are in order. First of all, Theorem \ref{thm:mainTauberianFinI} holds for $(\mathrm{Fin}, \mathcal{I})$-regular matrices, hence in particular it holds for regular matrices. Secondly, all $F_{\sigma\delta}$-ideals satisfy the above hypothesis, see Remark \ref{rmk:caseinclusions} below. 
In particular, Theorem \ref{thm:mainTauberianFinI} holds for 
the ideal $\mathcal{Z}$ of asymptotic density zero sets, 
for the ideals generated by nonnegative regular matrices (see \cite[Proposition 13]{MR3405547}),
for the ideal $\mathrm{Fin}$ (hence, providing another proof of Theorem \ref{thm:keoghpetersen}), and for the ideal of Banach density zero sets. 
Lastly, if $x$ is convergent then $\Sigma_{x,A}(\mathcal{I})=\Sigma$. 

Putting everything together, we obtain:
\begin{cor}\label{cor:characterization}
A real sequence $x$ is convergent if and only if there exists a regular matrix $A$ such that $\{\sigma \in \Sigma: A\sigma(x) \text{ is statistically convergent}\,\}$ is not meager. 
\end{cor}

In particular, letting $x$ be the sequence $(0,1,0,1,\ldots)$,  
we obtain that $\{\sigma(x): \sigma \in \Sigma\}=\{0,1\}^\mathbf{N}$ and $\Sigma_{x,A}(\mathcal{Z})\neq \Sigma$, which give us a stronger version of Steinhaus' theorem:
\begin{cor}\label{cor:steinhaus}
For each regular matrix $A$, there exists a $\{0,1\}$-valued sequence $x$ such that $Ax$ is not statistically convergent.
\end{cor}

On a similar direction, letting $A$ be the infinite identity matrix, it follows that, a real sequence $x$ is convergent if and only if $\{\sigma \in \Sigma: \sigma(x) \in c(\mathcal{I})\}$ is not meager, provided that $\mathcal{I}$ is an ideal on $\mathbf{N}$ as in Theorem \ref{thm:mainTauberianFinI}. For the latter class of ideals, this gives us a generalization of \cite[Corollary 2.7]{MR4165727}, which states that $\lim x=\eta$ if and only if $\{\sigma \in \Sigma: \mathcal{I}\text{-}\lim \sigma(x)=\eta\}$ is not meager, provided that $\mathcal{I}$ is a meager ideal and $\mathcal{I}\text{-}\lim x=\eta$ (note that $\eta$ is explicit). 
Other Tauberian theorems related to statistical convergence can be found in \cite{
MR4052262, 
MR2324333, MR1659877, MR1653457, MR1002541, MR1260176}. 

Finally, note that Theorem \ref{thm:mainTauberianFinI} does not hold without any restriction on the ideal $\mathcal{I}$. 
Indeed, if $A$ is a $(\mathrm{Fin},\mathcal{I})$-regular matrix, then it satisfies condition \ref{item:R1} of Theorem \ref{thm:connotleofinJ}, so that $A \in (\ell_\infty,\ell_\infty)$. 
Thus, if $x\in \ell_\infty\setminus c$ and $\mathcal{I}$ is a maximal ideal, then every subsequence $\sigma(x)$ is bounded, hence $A\sigma(x) \in c_b(\mathcal{I})$. 
In particular, for each bounded divergent sequence $x$, there exist an ideal $\mathcal{I}$ and a $(\mathrm{Fin}, \mathcal{I})$-regular matrix $A$ such that $\Sigma_{x,A}(\mathcal{I})$ is not meager.

\section{Preliminaries}\label{sec:preliminaries}

Given an ideal $\mathcal{I}$, let $G(\mathcal{I})$ be the following game introduced by Laflamme in \cite{MR1367134}: at stage $n \in \mathbf{N}$ player I chooses a set $A_n \in \mathcal{I}^\star$ and, then, player II chooses a nonempty finite set $F_n\subseteq A_n$. At the end of the game, player II is declared the winner if $\bigcup_n F_n \notin \mathcal{I}$. 
Moreover, a sequence $(F_k)$ of nonempty finite sets is said to be a $\mathcal{I}^\star$\emph{-universal set} if each $A \in \mathcal{I}^\star$ contains some $F_k$. We say that $\mathcal{I}^\star$ is $\omega$\emph{-diagonalizable by} $\mathcal{I}^\star$\emph{-universal sets} if there exists an infinite matrix $(F_{n,k})$ of nonempty finite sets such that each row is a $\mathcal{I}^\star$-universal set and, moreover, for each $A \in \mathcal{I}^\star$ there exist $n,m \in \mathbf{N}$ such that $F_{n,k} \cap A\neq \emptyset$ for all $k>m$. 
\begin{thm}\label{thm:laflamme}
Let $\mathcal{I}$ be an ideal on $\mathbf{N}$. Then the following are equivalent:
\begin{enumerate}[label={\rm (\textsc{c}\arabic{*})}]
\item \label{item:c1} $\mathcal{I}^\star$ is $\omega$-diagonalizable by $\mathcal{I}^\star$-universal sets\textup{;}
\item \label{item:c2} player II has a winning strategy in the game $G(\mathcal{I})$. 
\end{enumerate}
If, in addition, $\mathcal{I}$ is a Borel ideal, then they are also equivalent to\textup{:}
\begin{enumerate}[label={\rm (\textsc{c}\arabic{*})}]
\setcounter{enumi}{2}
\item \label{item:c3} $\mathcal{I}$ does not contain an isomorphic copy of $\mathrm{Fin}\times \mathrm{Fin}$\textup{;}
\item \label{item:c4} $\mathcal{I}$ is $F_\sigma$-separated from its dual filter $\mathcal{I}^\star$ \textup{(}that is, there exists an $F_\sigma$-set $\mathcal{K}$ such that $\mathcal{I}\subseteq \mathcal{K}$ and $\mathcal{K}\cap \mathcal{I}^\star=\emptyset$\textup{)}.
\end{enumerate}
\end{thm}
\begin{proof}
See \cite[Theorem 2.16(ii)]{MR1367134} for \ref{item:c1} $\Longleftrightarrow$ \ref{item:c2}. For the other equivalences, see \cite{MR2491780}.
\end{proof}

\begin{rmk}\label{rmk:caseinclusions}
As it has been shown in \cite[Corollary 1.5]{MR1758325}, all $F_{\sigma\delta}$-ideals satisfy condition \ref{item:c4}. Related results on condition \ref{item:c4} can be found in \cite[Proposition 3.6]{MR3090420} and \cite[Theorem 2.1]{MR3968131}.
\end{rmk}

At this point, for each ideal $\mathcal{I}$ and topological space $X$, define the $\mathcal{I}$-Baire one class 
$$
\mathcal{B}_1^{\mathcal{I}}(X):=\left\{f \in \mathbf{R}^X: \exists (f_n) \in C(X)^{\mathbf{N}}, \forall x \in X,\,\, f(x)=\mathcal{I}\text{-}\lim f_n(x)\right\},
$$
where $C(X)$ stands for the space of real-valued continuous functions on $X$. Note that $\mathcal{B}_1^{\mathcal{I}}(X)$ coincides with the classical Baire one class $\mathcal{B}_1(X)$ if $\mathcal{I}=\mathrm{Fin}$. However, it has been shown in \cite[Proposition 8]{MR2491780}, the same holds if $X$ is a complete metric space and $\mathcal{I}$ is an ideal for which player II has a winning strategy in the game $G(\mathcal{I})$, cf. also \cite{MR2520152} for related results. 
Some years later, 
Filip\'{o}w and Szuca proved in \cite{MR2899832}, in particular, that the hypothesis of completeness on $X$ is unnecessary to obtain the latter conclusion:
\begin{thm}\label{thm:filipowszuca}
Let $\mathcal{I}$ be an ideal on $\mathbf{N}$ such that $\mathcal{I}^\star$ is $\omega$-diagonalizable by $\mathcal{I}^\star$-universal sets. Moreover, let $X$ be a perfectly normal topological space. Then $\mathcal{B}_1^{\mathcal{I}}(X)=\mathcal{B}_1(X)$.
\end{thm}
\begin{proof}
See \cite[Theorem 3.2]{MR2899832}.
\end{proof}
This 
has been also obtained in \cite[Corollary 12]{MR2491780} for nonpathological analytic P-ideals $\mathcal{I}$ and arbitrary topological spaces $X$. 
We recall also the classical Baire classification theorem.
\begin{thm}\label{thm:baireclassification}
Let $X$ be a metrizable space and fix a Baire one function $f \in \mathcal{B}_1(X)$. Then the set of points of continuity of $f$ is a comeager $G_\delta$-set.
\end{thm}
\begin{proof}
See \cite[Theorem 24.14]{MR1321597}.
\end{proof}

Finally, we will use the well-known characterization of meager ideals due to Talagrand. 
\begin{thm}\label{prop:talagrand}
Let $\mathcal{I}$ be an ideal on $\mathbf{N}$. Then the following are equivalent: 
\begin{enumerate}[label={\rm (\textsc{m}\arabic{*})}]
\item \label{item:m1} $\mathcal{I}$ is meager\textup{;}
\item \label{item:m2} there exists $\sigma \in \Sigma$ such that $A \notin \mathcal{I}$ whenever $\mathbf{N} \cap [\sigma(n),\sigma(n+1)) \subseteq A$ for infinitely many $n \in \mathbf{N}$\textup{;}
\item \label{item:m3} $\mathcal{I}$ is $F_\sigma$-separated from the Fr\'{e}chet filter $\mathrm{Fin}^\star$ \textup{(}that is, there exists a sequence $(F_n)$ of closed sets such that $\mathcal{I}\subseteq \bigcup_n F_n$ and $F_n \cap \mathrm{Fin}^\star=\emptyset$ for all $n\in \mathbf{N}$\textup{)}\textup{.}
\end{enumerate}
\end{thm}
\begin{proof}
See \cite[Proposition 3.1]{MR4165727} and \cite[Theorem 2.1]{MR579439}.
\end{proof}

At this point, we split the intermediate results into two cases: the first one assumes that $x$ is a bounded divergent sequence, the second one that $x$ is unbounded.

\subsection{Bounded case}

\begin{lem}\label{lem:continuity}
Fix two sequences $x \in \ell_\infty$ and $a \in \ell_1$. 
Then the  map 
$$
\textstyle 
\Sigma\to \mathbf{R}: \sigma\mapsto a\cdot \sigma(x)
$$
is uniformly continuous. 
\end{lem}
\begin{proof}
The claim holds trivially for $x=0$, hence suppose hereafter that $x\neq 0$, so that $\|x\|>0$. 
Note that the topology on $\Sigma$ is metrizable by $d: \Sigma\times \Sigma \to \mathbf{R}$ defined as
$$
\forall \sigma_1, \sigma_2 \in \Sigma, \quad d(\sigma_1,\sigma_2)=\sum_{i \in\, \mathrm{Im}(\sigma_1) \bigtriangleup \,\mathrm{Im}(\sigma_2)}\frac{1}{2^i}.
$$
Fix $\varepsilon>0$. Since $a \in \ell_1$, there exists $k_0 \in \mathbf{N}$ such that $\sum_{k> k_0}|a_k|<\frac{\varepsilon}{2\|x\|}$. 
At this point, define $\delta:=1/2^{k_0}$ and note that, if there exists $k\le k_0$ such that $k \in \mathrm{Im}(\sigma_1) \bigtriangleup \mathrm{Im}(\sigma_2)$ then $d(\sigma_1,\sigma_2) \ge 1/2^k\ge \delta$. Therefore, for each $\sigma_1,\sigma_2 \in \Sigma$ with $d(\sigma_1,\sigma_2)<\delta$, we obtain that the least element of 
$\mathrm{Im}(\sigma_1) \bigtriangleup \mathrm{Im}(\sigma_2)$ is greater than $k_0$, which implies that 
$$
\textstyle 
|a\cdot \sigma_1(x)-a\cdot \sigma_2(x)|
=\left|\sum_{k>k_0}a_k(x_{\sigma_1(k)}-x_{\sigma_2(k)}))\right|
\le 2\|x\|\sum_{k>k_0}|a_k|<\varepsilon.
$$
This concludes the proof.
\end{proof}


\begin{cor}\label{cor:continuitypartialmaps}
Fix a matrix $A \in (c_0, \ell_\infty)$ and a sequence $x \in \ell_\infty$. Then the map 
$$
\textstyle 
\Sigma \to \mathbf{R}:\sigma \mapsto \sum_k a_{n,k}x_{\sigma(k)}
$$
is uniformly continuous for each $n \in \mathbf{N}$.
\end{cor}
\begin{proof}
Thanks to \cite[Theorem 2.3.5]{MR1817226}, $A$ belongs to $(c_0, \ell_\infty)$ if and only if $\sup_n \sum_k |a_{n,k}|<\infty$. The claim follows by Lemma \ref{lem:continuity}.
\end{proof}


\begin{lem}\label{lem:discontinuousverythwre}
Fix a bounded divergent sequence $x \in \ell_\infty\setminus c$, an ideal $\mathcal{I}$ on $\mathbf{N}$, and $(\mathrm{Fin}, \mathcal{I})$-matrix $A$. 
Then $\Sigma_{x,A}(\mathcal{I})$ is dense and the map 
\begin{equation}\label{eq:mapTj}
\textstyle 
T: \Sigma_{x,A}(\mathcal{I}) \to \mathbf{R}: \sigma \mapsto \mathcal{I}\text{-}\lim_n \sum_k a_{n,k}x_{\sigma(k)}
\end{equation}
is everywhere discontinuous.
\end{lem}
\begin{proof}
Let $B\subseteq \Sigma$ an arbitrary open ball. Observe that there exist $\sigma_1,\sigma_2\in B$ such that $\lim \sigma_1(x)=\alpha$ and $\lim \sigma_2(x)=\beta$, where $\alpha:=\limsup_n x_n$ and $\beta:=\liminf_n x_n$ are finite and distinct since $x \in \ell_\infty\setminus c$. 
Considering that $A$ is $(\mathrm{Fin}, \mathcal{I})$-regular, we obtain that $\alpha=\mathcal{I}\text{-}\lim A\sigma_1(x)$ and $\beta=\mathcal{I}\text{-}\lim A\sigma_2(x)$. 
In particular, $B\cap \Sigma_{x,A}(\mathcal{I})\neq \emptyset$, hence $\Sigma_{x,A}(\mathcal{I})$ is dense. 

In addition, $\sup_{\sigma, \sigma^\prime \in B \cap \Sigma_{x,A}(\mathcal{I})}|T(\sigma)-T(\sigma^\prime)|\ge \alpha-\beta$ for every open ball $B\subseteq \Sigma$, therefore $T$ cannot be continuous at any point of $\Sigma_{x,A}(\mathcal{I})$. 
\end{proof}


\begin{rmk}\label{rmk:borelmeasurability}
With the same hypotheses of Lemma \ref{lem:discontinuousverythwre}, assume that $\mathcal{I}$ is a Borel ideal. 
Then $\Sigma_{x,A}(\mathcal{I})$ is Borel and the map $T$ defined in \eqref{eq:mapTj} is Borel measurable. 
For, note that necessarily $A \in (c_0, \ell_\infty)$. Hence the map $T$ is the $\mathcal{I}$-pointwise limit of the sequence of functions $T_n: \sigma \mapsto \sum_k a_{n,k}x_k$ restricted to $\Sigma_{x,A}(\mathcal{I})$, which are continuous for each $n \in \mathbf{N}$ thanks to Corollary \ref{cor:continuitypartialmaps}. In particular, each of them is Borel measurable. The claim follows by \cite[Lemma 1 and Lemma 2]{MR2390885}. 
However, as  remarked in \cite[Example 1]{MR2390885}, some assumptions on $\mathcal{I}$ are needed to ensure that the $\mathcal{I}$-pointwise limit of measurable functions is still measurable.
\end{rmk}

\begin{thm}\label{thm:conclusionmaincategorybounded}
Fix a bounded divergent sequence $x \in \ell_\infty\setminus c$, an ideal $\mathcal{I}$ on $\mathbf{N}$ such that $\mathcal{I}^\star$ is $\omega$-diagonalizable by $\mathcal{I}^\star$-universal sets, and a $(\mathrm{Fin}, \mathcal{I})$-regular matrix $A$. 

Then $\Sigma_{x,A}(\mathcal{I})$ is meager. 
\end{thm}
\begin{proof}
Set $X:=\Sigma_{x,A}(\mathcal{I})$ and let $T$ be the map defined in \eqref{eq:mapTj}. 
Thanks to Corollary \ref{cor:continuitypartialmaps},  $T \in \mathcal{B}_1^{\mathcal{I}}(X)$. 
Since $X$ is metrizable and $\mathcal{I}^\star$ is $\omega$-diagonalizable by $\mathcal{I}^\star$-universal sets, then $T \in \mathcal{B}_1(X)$ by Theorem \ref{thm:filipowszuca}. 
At this point, it follows by Theorem \ref{thm:baireclassification} that the set of continuity points of $T$ is a comeager subset of $X$. 
Hence, by Lemma \ref{lem:discontinuousverythwre}, $X$ is a dense subset of $\Sigma$ which is meager in itself. 
We conclude that $X$ is meager in $\Sigma$. 
\end{proof}



\subsection{Unbounded case.}


\begin{thm}\label{lem:domaindefinition}
Fix sequences $x \in \omega\setminus \ell_\infty$ and $a \in \omega\setminus c_{00}$. Then 
$$
\textstyle 
\left\{\sigma \in \Sigma: \left(\sum_{k\le n} a_k x_{\sigma(k)}\right) \in \ell_\infty\right\}
$$ 
is meager. 
\end{thm}
\begin{proof}
Let $E$ be the claimed set. Observe that $E=\bigcup_{m}E_m$, where 
$$
\textstyle 
\forall m \in \mathbf{N}, \quad E_m:=\left\{\sigma \in \Sigma: \left|\sum_{k\le n} a_k x_{\sigma(k)}\right|\le m \text{ for all }n\right\}.
$$
We are going each $E_{m}$ is nowhere dense. Note that each $E_m$ is closed. Thus, let us assume for the sake of contradiction that there exists $m_0$ such that $E_{m_0}$ has nonempty interior, hence there exist positive integers $t_1<\cdots<t_j$ such that $\sigma \in E_{m_0}$ whenever $\sigma(s)=t_s$ for all $s=1,\ldots,j$. 
Since $a \notin c_{00}$, there exists a minimal integer $i_0\ge j+1$ such that $a_{i_0}\neq 0$.  
In addition, since $x \notin \ell_\infty$, there exists an integer $t_0 \ge t_j+i_0$ such that
\begin{equation}\label{eq:contradictionabsolutevalues}
\textstyle 
|a_{i_0}x_{t_0}|\ge m_0+1+\left|\sum_{k\le j} a_k x_{t_k}\right|.
\end{equation}

Finally, define  $\sigma_0: \mathbf{N}\to \mathbf{N}$ as follows:
\begin{enumerate}[label=(\roman*)]
\item \label{item:sigma1} $\sigma_0(s)=t_s$ for all $s=1,\ldots,j$;
\item \label{item:sigma2} $\sigma_0(j+s)=t_j+s$ for all $s=1,\ldots, i_0-j-1$;
\item \label{item:sigma3} $\sigma_0(i_0+s)=t_0+s$ for all integers $s\ge 0$.
\end{enumerate}
Note that $\sigma_0$ is strictly increasing, hence $\sigma_0 \in \Sigma$. Moreover, thanks to \ref{item:sigma1}, $\sigma_0 \in E_{m_0}$. On the other hand, by \ref{item:sigma3} we have $\sigma_0(i_0)=t_0$. And by the minimality of $i_0$ we get $a_{j+s}=0$ for all $s=1,\ldots,i_0-j-1$. Putting everything together with \eqref{eq:contradictionabsolutevalues}, we obtain that 
$$
\textstyle 
\left|\sum_{k\le i_0}a_k x_{\sigma_0(k)}\right|=\left|a_{i_0}x_{t_0}+\sum_{k\le j}a_kx_{t_k}\right|\ge m_0+1,
$$
which would imply that $\sigma_0 \notin E_{m_0}$. This contradiction concludes the proof. 
\end{proof}

A matrix $A$ is said \emph{row finite} if every row is in $c_{00}$, that is, for each $n$ there exists $k_0 \in \mathbf{N}$ such that $a_{n,k}=0$ for all $k\ge k_0$.
\begin{cor}\label{cor:unboundedcase1}
Let $A$ be a matrix which is not row finite, and fix $x \in \omega \setminus \ell_\infty$. 

Then 
$
\textstyle 
\{\sigma \in \Sigma: \sigma(x) \in \omega_A\}
$ 
is meager.
\end{cor}
\begin{proof}
By hypothesis, there exists $r_0 \in \mathbf{N}$ such that $a_{r_0,k}\neq 0$ for infinitely many $k$. 
Hence
\begin{displaymath}
\begin{split}
\textstyle
\{\sigma \in \Sigma: \sigma(x) \in \omega_A\}&\textstyle = \bigcap_r \left\{\sigma \in \Sigma: \left(\sum_{k\le n}a_{r,k}x_{\sigma(k)}\right) \in c\right\}\\
&\textstyle 
\subseteq 
\left\{\sigma \in \Sigma: \left(\sum_{k\le n}a_{r_0,k}x_{\sigma(k)}\right) \in \ell_\infty\right\}.
\end{split}
\end{displaymath}
The conclusion follows by Theorem \ref{lem:domaindefinition}.
\end{proof}

\begin{thm}\label{thm:unboundedrowfinite}
Let $\mathcal{I}$ be a meager ideal on $\mathbf{N}$. Fix an unbounded sequence $x \in \omega \setminus \ell_\infty$ and a row finite matrix $A$ such that 
\begin{equation}\label{eqitem:U1}
\textstyle 
\forall w \in \mathbf{N}, \quad 
Z_w:=\{n \in \mathbf{N}: a_{n,k}=0 \text{ for all }k\ge w\} \in \mathcal{I}.
\end{equation}

Then 
$
\textstyle 
\{\sigma \in \Sigma: A\sigma(x) \in \ell_\infty(\mathcal{I})\}
$  
is meager. 
\end{thm}
\begin{proof} 
Thanks to Theorem \ref{prop:talagrand}, there exists a sequence $(F_n)$ of closed sets in $\mathcal{P}(\mathbf{N})$ such that $\mathcal{I}\subseteq \bigcup_n F_n$ and $F_n \cap \mathrm{Fin}^\star=\emptyset$ for all $n\in \mathbf{N}$. 
Fix $w \in \mathbf{N}$. 
Since $Z_w\in\mathcal{I}$ by condition \eqref{eqitem:U1}, it follows that the family $\mathcal{J}_w:=\{U\setminus Z: U\in \mathcal{I}\}$ is an ideal on $T_w:=\mathbf{N}\setminus Z_w$ such that $\mathcal{J}_w\subseteq \bigcup_n G_n$ and $G_n \cap \mathrm{Fin}^\star=\emptyset$ for all $n \in \mathbf{N}$, where $G_n:=F_n \cap \mathcal{P}(T_w)$ is closed in $\mathcal{P}(T_w)$. 
It follows, again by Theorem \ref{prop:talagrand}, that $\mathcal{J}_w$ is a meager ideal on $T_w$. 
Hence there exists a partition $\{I_{w,1}, I_{w_2},\ldots\}$ of $T_w$ into nonempty finite subsets such that a set $U\subseteq T_w$ does not belong to $\mathcal{J}_w$ whenever $I_{w,n} \subseteq U$ for infinitely many $n$. 

At this point, let $S$ be the claimed set. Since $T_w \in \mathcal{I}^\star$ for each $w \in \mathbf{N}$, we obtain that
\begin{displaymath}
\begin{split}
\textstyle 
\forall w \in \mathbf{N}, \quad 
S&\textstyle =\bigcup_m\{\sigma \in \Sigma: \{n \in \mathbf{N}: |\sum_ka_{n,k}x_{\sigma(k)}|\ge m\}\in \mathcal{I}\}\\
&\textstyle 
=\bigcup_m\{\sigma \in \Sigma: \{n \in T_w: |\sum_{k} a_{n,k}x_{\sigma(k)}|\ge m\}\in \mathcal{J}_w\}\\
&\textstyle 
\subseteq \bigcup_m\{\sigma \in \Sigma: \{n \in I_{w,q}: |\sum_k a_{n,k}x_{\sigma(k)}|\ge m\}\neq I_{w,q} \text{ for all large }q\}\\
&\textstyle 
= \bigcup_m\bigcup_p \bigcap_{q\ge p}\{\sigma \in \Sigma: \{n \in I_{w,q}: |\sum_ka_{n,k}x_{\sigma(k)}|\ge m\}\neq I_{w,q}\}\\
&\textstyle 
= \bigcup_m\bigcup_p \bigcap_{q\ge p} \bigcup_{J\subsetneq I_{w,q}} \{\sigma \in \Sigma: \{n \in I_{w,q}: |\sum_ka_{n,k}x_{\sigma(k)}|\ge m\}=J\}.
\end{split}
\end{displaymath}
Note that all the inner sets are closed, so that also each
$$
\textstyle 
S_w(m,p):=\bigcap_{q\ge p} \{\sigma \in \Sigma: \forall q\ge p, \exists n \in I_{w,q}, |\sum_ka_{n,k}x_{\sigma(k)}|< m\}
$$
is closed. Since $S\subseteq \bigcup_{m,p}S_w(m,p)$, it is sufficient to show that there exists $w \in \mathbf{N}$ such that each $S_w(m,p)$ has empty interior: this would imply that $S$ is contained in a countable union of nowhere dense sets. 

To this aim, suppose for the sake of contradiction that there exist $m_0,p_0 \in \mathbf{N}$ and positive integers $t_{1}<\cdots<t_{j_0}$ such that $\sigma \in S(m_0,p_0)$ whenever $\sigma(s)=t_{s}$ for all $s=1,\ldots,j_0$. 
Set $w_0:=j_0+1$. 
By condition \eqref{eqitem:U1}, the set $T_{w_0}$ belongs to $\mathcal{I}^\star$, hence it is nonempty. 
Note that, by construction, $r_n \ge w_0$ for all $n\in T_{w_0}$. 
Define $n_0:=\min T_{w_0}$, let $p_1$ be the positive integer for which $n_0 \in I_{w_0,p_1}$, and finally set
$$
\textstyle
\alpha:=\min\{|a_{n,k}|: a_{n,k} \neq 0 \text{ and }n \in I_{w_0,q_0}\}, \,\,
\text{ where }q_0:=\max\{p_0, p_1+1\}.
$$
Note that $\alpha$ is well defined since $A$ is row finite, and set $k_0:=\max\{r_n: n \in I_{w_0,q_0}\}$. 


Finally, define $\sigma_0: \mathbf{N}\to \mathbf{N}$ recursively as follows:
\begin{enumerate}[label=\textup{(}\roman*\textup{)}]
\item \label{item:cond1} $\sigma_0(s)=t_s$ for all $s=1,\ldots,j_0$;
\item \label{item:cond2} for each $s=j_0+1,\ldots,k_0$, if $\sigma_0(1)<\cdots<\sigma_0(s-1)$ are already defined, then $\sigma_0(s)$ is an integer $h>\sigma_0(s-1)$ such that
$$
\textstyle 
|x_h|\ge \frac{1}{\alpha}\left(m_0+\max\left\{\left|\sum_{k< s}a_{n,k}x_{\sigma_0(k)}\right|: n \in I_{w_0,q_0}\right\}\right).
$$
\item \label{item:cond3} $\sigma_0(k_0+s)=\sigma_0(k_0)+s$ for all $s \in \mathbf{N}$.
\end{enumerate}
Considering that $\sigma_0$ is strictly increasing by construction, it follows by \ref{item:cond1} that $\sigma_0 \in S_{w_0}(m_0,p_0)$. 
At the same time, since $j_0<r_n\le k_0$ for all $n \in I_{w_0,q_0}$, we obtain by \ref{item:cond2} that
\begin{displaymath}
\begin{split}
\textstyle 
\forall n\in I_{w_0,q_0}, \quad |\sum_ka_{n,k}x_{\sigma_0(k)}|
&\textstyle =|\sum_{k\le r_n}a_{n,k}x_{\sigma_0(k)}|\\
&\textstyle \ge |a_{n,r_n} x_{\sigma_0(r_n)}|-\left|\sum_{k< r_n}a_{n,k}x_{\sigma_0(k)}\right|\\
&\textstyle \ge \alpha |x_{\sigma_0(r_n)}|- \max\left\{\left|\sum_{k< r_n}a_{i,k}x_{\sigma_0(k)}\right|: i \in I_{w_0,q_0}\right\}\ge m_0.
\end{split}
\end{displaymath}
This implies that $\sigma_0 \notin S_{w_0}(m_0,p_0)$, hence we obtained the desired contradiction. 
\end{proof}

\section{Proof of Theorem \ref{thm:mainTauberianFinI}}\label{sec:mainproof}

\begin{proof}[Proof of Theorem \ref{thm:mainTauberianFinI}] 
First, suppose that $x$ is a bounded divergent sequence. Then $\Sigma_{x,A}(\mathcal{I})$ is meager by Theorem \ref{thm:laflamme} and Theorem \ref{thm:conclusionmaincategorybounded}. 

Secondly, suppose that $x$ is unbounded. If $A$ is not row finite, then $\Sigma_{x,A}(\mathcal{I})\subseteq \{\sigma \in \Sigma: \sigma(x) \in \omega_A\}$, which is meager by Corollary \ref{cor:unboundedcase1}. 
Otherwise, suppose hereafter that $A$ is a row finite $(\mathrm{Fin}, \mathcal{I})$-regular matrix. 
Note that $\mathcal{I}$ is a Borel ideal, hence it is meager. 
Thanks to Theorem \ref{thm:connotleofinJ}, $A$ satisfies conditions \ref{item:R1}--\ref{item:R3}. 
For each $w \in \mathbf{N}$, it follows by \ref{item:R2} and \ref{item:R3} that $\mathcal{I}\text{-}\lim_n\sum_{k\ge w}a_{n,k}=1$. 
In particular, $Z_w$ is contained in $\{n \in \mathbf{N}: \sum_{k\ge w}a_{n,k}=0\}$, which belongs to $\mathcal{I}$. 
This implies that condition \eqref{eqitem:U1} in Theorem \ref{thm:unboundedrowfinite} holds. 
It follows that $\Sigma_{x,A}(\mathcal{I})$ is contained in $\{\sigma \in \Sigma: A\sigma(x) \in \ell_\infty(\mathcal{I})\}$, which is meager by Theorem \ref{thm:unboundedrowfinite}. 
\end{proof}

\section{Concluding remarks}

It doesn't come as a surprise that, under suitable hypotheses on the matrix $A$ and the ideal $\mathcal{I}$, the set $\Sigma_{x,A}(\mathcal{I})$ is either meager or the whole $\Sigma$. Indeed, by a known topological $0$-$1$ law, see e.g. \cite[Theorem 8.47]{MR1321597}, a tail subset of $\Sigma$ with the Baire property is either meager or comeager. This applies also in our case. For, let $\mathcal{I}$ be a Borel ideal and $A$ be a $(\mathrm{Fin}, \mathcal{I})$-regular matrix. Then, it follows by Remark \ref{rmk:borelmeasurability} that $\Sigma_{x,A}(\mathcal{I})$ is Borel, hence it has the Baire property. Moreover, if $\sigma_1(n)=\sigma_2(n)$ for all but finitely many $n$, then $y\in c_0$ where $y_n:=x_{\sigma_2(n)}-x_{\sigma_1(n)}$ for all $n$. Hence, if $\sigma_1 \in \Sigma_{x,A}(\mathcal{I})$ then 
$$
\textstyle 
\mathcal{I}\text{-}\lim A\sigma_2(x)=\mathcal{I}\text{-}\lim A\sigma_1(x)+\mathcal{I}\text{-}\lim Ay=\mathcal{I}\text{-}\lim A\sigma_1(x).
$$
This implies that also $\sigma_2 \in \Sigma_{x,A}(\mathcal{I})$, proving that $\Sigma_{x,A}(\mathcal{I})$ is also a tail set. 

In the same direction of \cite{MR404918, MR382890, MR425410}, we leave as open question for the interest reader to check whether the analogues of Theorem \ref{thm:mainTauberianFinI} hold for permutations and strechings of the sequence $x$.

Finally, our main result seems to be related also to a conjecture of DeVos \cite{MR501124} which can be reformulated as follows: if $E$ is an FK-space (that is, a locally convex vector space of $\omega$ which is also Fr\'{e}chet and with continuous coordinates) containing $c_{00}$, then $\{0,1\}^{\mathbf{N}}\subseteq E$ if and only if $E \cap \{0,1\}^{\mathbf{N}}$ is not meager. However, it seems quite unlikely that spaces of the type $\{x \in \omega: Ax \in c(\mathcal{I})\}$ may provide a counterexample to the latter conjecture. Indeed,  it has been shown by Connor in \cite[Theorem 3.3]{MR954458} that, even for the well-behaved ideal $\mathcal{I}=\mathcal{Z}$, the unique FK-space containing $c(\mathcal{Z})$ is $\omega$.



\bibliographystyle{amsplain}

\end{document}